\documentclass{article}
\usepackage{mathtext}
\usepackage{fontenc}
\usepackage{amssymb}
\usepackage{amsmath}
\usepackage{amsthm}
\usepackage{floatflt}
\usepackage{wrapfig}
\usepackage{caption2}
\usepackage{comment}
\usepackage{graphicx}
\usepackage{multirow,mathdots}
\usepackage{ulem}
\usepackage{tikz}
\usetikzlibrary{positioning, arrows, shapes, decorations.pathreplacing, calc}
\usetikzlibrary{matrix}

\tikzset{x=0.9em,y=1.1em} % With these parameters we regulate the cell step horizontally and vertically.
\def\ccolor#1{ ++(1,0) node [rectangle,  minimum size=1] {\raisebox{-0.25em}[0.2em][0em]{\makebox[0mm][c]{#1}}} }

\usepackage{cmap}
\usepackage[unicode, pdftex]{hyperref}

\newtheorem{theorem}{Theorem}
\newtheorem{lemma}{Lemma}

\newtheorem{corollary}{Corollary}
\newtheorem{observation}{Observation}
\newtheorem{property}{Property}

\newtheorem{remark}{Remark}

\begin{document}

\author{\bf Yuriy Tarannikov\footnote{Sobolev Institute of Mathematics, Siberian Branch of the Russian Academy of Sciences;
Lomonosov Moscow State University; e-mail: yutarann@gmail.com \newline
Theorems \ref{some_tight_values} and \ref{some_asymptotics} were obtained with the financial support of the Ministry of Science and Higher Education of Russia in the framework of the program of the Moscow Center for Fundamental and Applied Mathematics (contract No. 075--15--2022--284), Theorems \ref{tight_values} and \ref{th_asymp} were obtained at the expense of the Russian Science Foundation (grant No. 22--11--00266), \url{https://rscf.ru/ en/project/22-11-00266/}.}}
\title{On the number of partitions of the hypercube ${\bf Z}_q^n$ into large subcubes}

\date{}
\maketitle

\begin{abstract}
We prove that the number of partitions of the hypercube ${\bf Z}_q^n$ into $q^m$ subcubes of dimension $n-m$ each for fixed $q$, $m$ and growing $n$ is asymptotically equal to
$n^{(q^m-1)/(q-1)}$.

For the proof, we introduce the operation of the bang of a star matrix and demonstrate that any star matrix, except for a fractal, is expandable under some bang, whereas a fractal remains to be a fractal under any bang.
\end{abstract}


MSC 05A18

\section{Introduction}

Let $q,m,n$ be integers, $q\ge 2$, $n\ge m\ge 0$.
The ($n-m$)-dimensional subcube in ${\bf Z}_q^n$ is a subset of ${\bf Z}_q^n$ such that some $m$ components are fixed, and each of the remaining $n-m$ components runs through all possible values from ${\bf Z}_q$.

In {\it partitioning into subcubes} each vector of ${\bf Z}_q^n$ must fall into exactly one subcube.
A partition into subcubes is called {\it Agievich-primitive}, or simply
{\it A-primitive}\footnote{Agievich in \cite{Agievich08} introduced the term ''primitive'' for this concept (in more general case of partitioning into affine subspaces), but the use of this term is controversial, since it rather characterizes a certain non-degeneracy of the partition. In addition, such non-degeneracy of the partition can be defined in different ways, and the word ''primitive'' is generally overloaded in mathematics. At the same time, giving another term also seems incorrect, therefore in \cite{Tar22} it was proposed to call such a partition {\it Agievich-primitive} or {\it A-primitive}. Note that in \cite{FHKIRSV23} the authors used the term ''{\it tight}'' for the same concept, which also does not seem to us to be ''tight''.}, if each component is fixed in at least one of the subcubes of the partition.

The most well-known problem is the partitioning problem into small-dimensional subcubes. Thus, if all subcubes of a partition of a Boolean cube have dimension $1$, then these subcubes are edges, and the partitions are called perfect matchings, and the problem of their number is well known. In \cite{ABP24}, the problems of partitioning a Boolean cube into subcubes are considered, mainly of small dimensions, which can also be different within a single partition.

Partitions into subcubes (not necessarily of the same dimension) with an additional irreducibility condition are studied in the paper \cite{FHKIRSV23}, motivated by the fact that partitions into subcubes in the binary case correspond to unsatisfiable Boolean CNFs of a special type. Note that the partitions into subcubes of codimension at most $k$, where $k$ is a constant, considered in our paper correspond in the above context to the unsatisfiable CNF of a special type from the $k$-SAT problem.

Partitions into one-dimensional subcubes in a not necessarily binary case are the subject of \cite{Pot11}.
Partitioning into subcubes of the same dimension
is a special case of $A(n,q,w,t)$-designs, the expression of the number of which through multidimensional permanents was considered in \cite{Pot14}.

The main subject of study in \cite{Tar22} was partitions into affine subspaces, and for partitions into subcubes, which are a special case of partitions into affine subspaces, the following statements were proved, oriented towards partitions into subcubes of the same large dimension.
\begin{theorem}\label{coord_thm}~\cite{Tar22}
Let $q\ge 2$. For any natural $m$ there exists the smallest natural $N$ such that for $n>N$ there are no A-primitive partitions of ${\bf Z}_q^n$ into $q^m$ subcubes of dimension $n-m$.
\end{theorem}

We introduce the following notation:
\begin{itemize}
\item{$N^{\rm coord}_q(m)$ --- the smallest $N$ from Theorem \ref{coord_thm};}
\item{$c^{\rm coord}_q(n,m)$ --- the number of different unordered partitions of ${\bf Z}_q^n$ into $q^m$ subcubes of dimension $n-m$;}
\item{$c^{\rm coord*}_q(n,m)$ --- the number of different unordered A-primitive partitions of ${\bf Z}_q^n$ into $q^m$ subcubes of dimension $n-m$.}
\end{itemize}

\begin{theorem}~\cite{Tar22}
The following formula is valid
\begin{equation}\label{eq:coord_formula_modified}
c^{\rm coord}_q(n,m)=\sum\limits_{h=m}^{N^{\rm coord}_q(m)} \binom{n}{h} c^{\rm coord*}_q(h,m).
\end{equation}
\end{theorem}

\begin{theorem}\label{general_form_of_asymptotics}~\cite{Tar22}
Let $q$ and $m$ be fixed, $n\to\infty$. Then the following asymptotics holds
\begin{equation*}
c^{\rm coord}_q(n,m)\sim C'n^{N^{\rm coord}_q(m)},
\end{equation*}
where $C'=\frac{c^{\rm coord*}_q(N^{\rm coord}_q(m),m)}{N^{\rm coord}_q(m)!}$.
\end{theorem}

Also in \cite{Tar22} the bounds $\frac{q^m-1}{q-1}\le N^{\rm coord}_q(m)\le m\cdot q^{m-1}$ were obtained and exact values $N^{\rm coord}_q(2)=q+1$ were established.

The main result of this paper is the following Theorem \ref{th_asymp}, a complete proof
\footnote{Note that in 2024 we sent the following Theorems \ref{some_tight_values} and \ref{some_asymptotics} without proofs to the conference ''The Problems of Theoretical Cybernetics''.

\begin{theorem}\label{some_tight_values}
The following tight values hold: $N^{\rm coord}_2(4)=15$, $N^{\rm coord}_2(5)=31$,
$N^{\rm coord}_q(3)=q^2+q+1$,
$c^{\rm coord*}_2(15,4)=15!$,
$c^{\rm coord*}_2(31,5)=31!$,
$c^{\rm coord*}_q(q^2+q+1,3)=(q^2+q+1)!$.
\end{theorem}

\begin{theorem}\label{some_asymptotics} For $n\to\infty$ the following asymptotics are valid: $c^{\rm coord}_2(n,4)\sim n^{15}$, $c^{\rm coord}_2(n,5)\sim n^{31}$, $c^{\rm coord}_q(n,3)\sim n^{q ^2+q+1}$.
\end{theorem}

The proofs of Theorems \ref{some_tight_values} and \ref{some_asymptotics}, originally obtained by a substantially different method, are not given in this paper, since their assertions are special cases of Theorems \ref{tight_values} and \ref{th_asymp}, respectively.} of which will be obtained at the end of the paper after introducing the necessary concepts and proving auxiliary assertions.

\begin{theorem}\label{th_asymp}
Let $q,m$ be fixed positive integer numbers, $q>1$. Then for $n\to\infty$ the following asymptotics holds:
\begin{equation}\label{eq:main_asymptotics}
c^{\rm coord}_q(n,m)\sim n^{\frac{q^m-1}{q-1}}.
\end{equation}
\end{theorem}

\section{Star matrices of partitions and their properties}

{\it The star pattern} of a subcube of ${\bf Z}_q^n$ is the vector of length $n$ over ${\bf Z}_q\cup \{*\}$, where the elements of
${\bf Z}_q$ correspond to fixed components, while $*$ corresponds to a
''free'' component.

For example, the vector $(0,*,1,0,*)$ is a star pattern of the following subcube in ${\bf Z}_2^5$:
\begin{equation*}
\begin{array}{ccc}
\{&(0,0,1,0,0),&\cr
 &(0,0,1,0,1),&\cr
&(0,1,1,0,0),&\cr
&(0,1,1,0,1)&\}.
\end{array}
\end{equation*}
The matrix whose rows contain the star patterns of all subcubes of a partition is called the {\it star matrix} of this partition.

For example, the star matrix
\begin{equation*}
\left(
\begin{array}{ccc}
0, & 0, &*\cr
0, & 1, &*\cr
1, & *, &0\cr
1, & *, &1
\end{array}
\right)
\end{equation*}
defines a partition of ${\bf Z}_2^3$ into $2^2$ subcubes of dimension $3-2=1$ each.

It is easy to see that a partition of ${\bf Z}_q^n$ into $q^m$ subcubes of the same dimension $n-m$ is defined by a star matrix with exactly $q^m$ rows and $n$ columns, and is A-primitive if and only if its star matrix does not contain a column of only $*$. The star matrix of an A-primitive partition will also be called {\it A-primitive.}

Note that a special case of partitioning into subcubes are {\it associative block designs (ABDs),} which were introduced by Rivest \cite{Riv1974} for use in hashing algorithms and studied in a number of papers (see, for example, \cite{Brou1978,vL1985,vLW2001,Pout1986}).
ABD is a partition of ${\bf Z}_2^n$ into subcubes of the same dimension with the additional requirement that each column of the partition matrix contains the same number of stars. It is obvious from the definition that an ABD is an A-primitive partition. In the works on ABDs, elements of the technique of studying star matrices were developed, which are also useful in a more general case than ABDs.

{\it The size of a column} of a star matrix is the quantity of numbers in it. {\it The overlap} of two columns is the set of rows in which these columns simultaneously contain numerical values, and the number of such rows is the {\it overlap size.} If the overlap size of two columns is zero, then we say that these columns {\it do not overlap,} and if the overlap size coincides with the size of one of the columns, then we say that the column of smaller size {\it is inserted} into the column of larger size.
A submatrix of width $1$, consisting of all the numbers included in the column, and only of them, is called the {\it rod} of the column.

The statement of the Theorem \ref{th_asymp} will be a consequence of the Theorem \ref{tight_values},
in which the exact values of the quantities $N^{\rm coord}_q(m)$ and
$c^{\rm coord*}_q\left(N^{\rm coord}_q(m),m\right)$ will be established.
In turn, to establish the values written out in the formulation of the Theorem \ref{tight_values},
an analysis of A-primitive star matrices
of size $q^m \times N^{\rm coord}_q(m)$ will be performed.
Let us formulate and prove several lemmas on the structure of star matrices that we need.

\begin{lemma}\label{lemma1}
In the star matrix of the partition of the hypercube ${\bf Z}_q^n$, for any two distinct rows, there is a column that has different values from ${\bf Z}_q$ in these rows.
\end{lemma}

\begin{proof}
Suppose the opposite. Let there be two distinct rows in the star matrix of the partition for which there is no column that has different values from ${\bf Z}_q$ in these rows. Then
we can replace the stars in these rows with numbers so that both rows become the same. The resulting
vector belongs to both subcubes defined by the star patterns of the rows,
therefore these subcubes intersect, which is impossible in the partition.
\end{proof}

\begin{lemma}\label{equality_in_single_column}
In the star matrix $M$ of the partition of the hypercube ${\bf Z}_q^n$ into subcubes of the same dimension, in any column all values from ${\bf Z}_q$ occur the same number of times.
\end{lemma}

\begin{proof}
Consider the $i$th column of the matrix $M$. If some row of the star matrix $M$ has $*$ in the $i$th column, then the corresponding subcube for each $a\in {\bf Z}_q$ contains exactly $q^{n-m-1}$ vectors with the value $a$ in the $i$th column.
If some row of the star matrix has $a$ in the $i$th column, $a\in {\bf Z}_q$, then
all $q^{n-m}$ vectors of the corresponding subcube have $a$ in the $i$th column.
Any vector of ${\bf Z}_q^n$ belongs to exactly one subcube of the partition. This implies the assertion of the Lemma \ref{equality_in_single_column}.
\end{proof}

\begin{lemma}\label{equality_for_sums by_modulo_in_two_columns}
Let $M$ be the star matrix of the partition of the hypercube ${\bf Z}_q^n$ into subcubes of the same dimension, $i$ and $j$ be the numbers of two of its distinct columns, and $\pi$ be some permutation from $S_q$.
Consider the set $R$ of all rows from $M$ that simultaneously have numerical values from ${\bf Z}_q$ in columns $i$ and $j$. Then the sum $(m_{r,i}+\pi(m_{r,j}))\pmod{q}$, taken over all rows $r$ from $R$, takes each value from ${\bf Z}_q$ the same number of times.
\end{lemma}

\begin{proof}
If for all vectors $a=(a_1,\dots,a_n)$ from ${\bf Z}_q^n$ we consider the sum $(a_i+\pi(a_j))\pmod{q}$, then this sum will obviously take each of the $q$ values from ${\bf Z}_q$ the same number of times.
If we consider such a sum for all vectors of a subcube whose star pattern contains a star in at least one of the two components $i$ and $j$, then this sum will also obviously take each of the $q$ values from ${\bf Z}_q$ the same number of times, because the vectors from the subcube are combined in this case into subsets of $q$ vectors that differ only in the component corresponding to the star, and the sum $(a_i+\pi(a_j))\pmod{q}$ on the vectors of the subset takes each of $q$ values once (because the sum modulo $q$ is a group operation). Consequently, for the set of subcubes corresponding to the rows from $R$, the sum $(a_i+\pi(a_j))\pmod{q}$ over all vectors of these subcubes must take each of the $q$
values from ${\bf Z}_q$ the same number of times. However, for vectors from the subcube corresponding to row $r$ of $R$, the sum $(a_i+\pi(a_j))\pmod{q}$ takes the same value, equal to $(m_{r,i}+\pi(m_{r,j}))\pmod{q}$, and by assumption, the subcubes contain the same number of vectors. This implies the assertion of the Lemma \ref{equality_for_sums by_modulo_in_two_columns}.
\end{proof}

Lemmas close to those given above, in some ways more general, in some ways more specific, are contained in
\cite{Riv1974,Brou1978,vL1985,vLW2001,Pout1986,Pot22}.

Naturally, we could have formulated our lemmas in a more general form (in particular, for collections of more than two columns), but we did not pursue this goal, limiting ourselves to the formulation and proof of the lemmas in the minimal form necessary for the completeness of the presentation of our subsequent results.

\section{An obvious lower asymptotic bound for the number of partitions}\label{sec:asymp_lower_bound}

The lower asymptotic bound
\begin{equation}\label{eq:asymp_lower_bound}
c^{\rm coord}_q(n,m)\ge n^{\frac{q^m-1}{q-1}}(1+o(1))
\end{equation}
is obvious. Note that it is not necessary to prove it separately, because the asymptotics
of the quantity $c^{\rm coord}_q(n,m)$, presented in the formula (\ref{eq:main_asymptotics}) of the Theorem
\ref{th_asymp}, will be obtained automatically from the Theorem \ref{general_form_of_asymptotics} after we establish the exact values of the quantities $N^{\rm coord}_q(m)=\frac{q^m-1}{q-1}$ and
$c^{\rm coord*}_q(\frac{q^m-1}{q-1},m)=\left(\frac{q^m-1}{q-1}\right)!$ in the Theorem \ref{tight_values}. However, we will now give some reasoning to demonstrate the validity of the bound
(\ref{eq:asymp_lower_bound}), since this reasoning will allow us to feel the meaning and importance of the fractal matrix, the definition of which we will give in the next section.

So, let us cut the hypercube ${\bf Z}_q^n$ along one of the coordinates into $q$
subcubes\footnote{If someone is confused by the fact that a single cut divides a hypercube into $q$ subcubes,
then one can imagine how after one cut of the spine of a book it breaks up into several blocks.}
of dimension $n-1$. This can be done in $n$ ways. We will cut each of the $q$ resulting subcubes in one of the $n-1$ ways along one of its remaining coordinates into subcubes of dimension $n-2$, and so on, until we get a partition into $q^m$ subcubes of dimension $n-m$ each. The process described above can be implemented in
$$
\prod\limits_{i=0}^{m-1}\prod\limits_{j=1}^{q^i}(n-i)\sim n^{\frac{q^m-1}{q-1}}
$$
ways.

Note that some of these methods will lead to cuts of different subcubes along one coordinate. It can be proved that the proportion of partitions with such repetitions will be asymptotically small. To do this, we will prohibit the partitioning of two different subcubes along one coordinate. All the same, we have a finite number of factors (because $q$ and $m$ are finite), and $n$ tends to infinity.

With such a restriction, the partitions will obviously be different, and their number is equal to
$$
\prod\limits_{i=1}^{\frac{q^m-1}{q-1}} (n-i+1)\sim n^{\frac{q^m-1}{q-1}},
$$
i.e., asymptotically, despite the restrictions, it will be the same.

Therefore, the validity of the bound (\ref{eq:asymp_lower_bound}) is established.

Note that not all partitions can be obtained using the procedures described above.
For example, in the partition given by the star matrix given in Rivest's paper \cite{Riv1974},
cutting the original hypercube along any coordinate will lead to cutting some subcubes of the partition, because each column contains stars (see~Fig.~\ref{fig:star_matrix_of_Rivest}). Therefore, the question of an upper bound for the quantity
$c^{\rm coord}_q(n,m)$ requires additional research, which we will conduct in this paper.

\begin{figure}[ht]
\centering
$$
\left(\begin{array}{cccc}
0&0&*&0\cr
1&0&0&*\cr
*&1&0&0\cr
1&*&1&0\cr
1&1&*&1\cr
0&1&1&*\cr
*&0&1&1\cr
0&*&0&1
\end{array}\right)
$$
\caption{Star matrix with stars in all columns.}
\label{fig:star_matrix_of_Rivest}
\end{figure}

\section{Fractal star matrices and their properties}

We introduce matrices $M_{q,m}$ recursively as follows. The matrix $M_{q,0}$ has one row and zero columns\footnote{If a matrix with one row and zero columns seems casuistic to someone, then one can start with the matrix $M_{q,1}$, which has $q$ rows and one column and is a vertical column of all the numbers from ${\bf Z}_q$ written out in ascending order.},
the matrix $M_{q,m}$, $m=1,\dots,$ is defined through the matrix $M_{q,m-1}$, as shown in
Fig.~\ref{fig:recursive_construction_of_fractal_matrix},
\begin{figure}[ht]
\centering
$$
M_{q,m} =
\begin{tabular}{|c|c|c|c|c|}
    \hline
    $0$ & \multirow{3}{*}{\large $M_{q,m-1}$}&
        \multirow{3}{*}{\large $\scalebox{2}{$*$}_{q,m-1}$}&
        \multirow{3}{*}{\large $\cdots$}&
        \multirow{3}{*}{\large $\scalebox{2}{$*$}_{q,m-1}$}\\
    $\vdots$ &&&&\\
    $0$ &&&& \\ \hline
    $1$ & \multirow{3}{*}{\large $\scalebox{2}{$*$}_{q,m-1}$}&
        \multirow{3}{*}{\large $M_{q,m-1}$}&
        \multirow{3}{*}{\large $\cdots$}&
        \multirow{3}{*}{\large $\scalebox{2}{$*$}_{q,m-1}$}\\
    $\vdots$ &&&&\\
    $1$ &&&& \\ \hline
     \multirow{3}{*}{\large $\vdots$}&
     \multirow{3}{*}{\large $\vdots$}&
        \multirow{3}{*}{\large $\vdots$}&
        \multirow{3}{*}{\large $\ddots$}&
        \multirow{3}{*}{\large $\vdots$}\\
    &&&&\\
    &&&&\\ \hline
    $q{-}1$ & \multirow{3}{*}{\large $\scalebox{2}{$*$}_{q,m-1}$}&
        \multirow{3}{*}{\large $\scalebox{2}{$*$}_{q,m-1}$}&
        \multirow{3}{*}{\large $\cdots$}&
        \multirow{3}{*}{\large $M_{q,m-1}$}\\
    $\vdots$ &&&&\\
    $q{-}1$ &&&& \\ \hline
\end{tabular}
$$
\caption{Recursive construction of a fractal matrix.}
\label{fig:recursive_construction_of_fractal_matrix}
\end{figure}
where $\scalebox{2}{$*$}_{q,m-1}$ is a matrix of the same size as $M_{q,m-1}$, but consisting only of stars.

We will call the matrices $M_{q,m}$ {\it fractal} star matrices, or simply
{\it fractals\footnote{This name was chosen because the matrices $M_{q,m}$ have some self-similarity.}.}

Let us give examples of fractal star matrices (see ~Fig.\ref{fig:examples_of_fractal_matrices}):

\begin{figure}[ht]
\centering
$$
M_{2,2}=\left(\begin{array}{ccc}
0&0&*\cr
0&1&*\cr
1&*&0\cr
1&*&1
\end{array}\right), \qquad
M_{2,3}=\left(\begin{array}{ccccccc}
0&0&0&*&*&*&*\cr
0&0&1&*&*&*&*\cr
0&1&*&0&*&*&*\cr
0&1&*&1&*&*&*\cr
1&*&*&*&0&0&*\cr
1&*&*&*&0&1&*\cr
1&*&*&*&1&*&0\cr
1&*&*&*&1&*&1
\end{array}\right), \qquad
M_{3,2}=\left(\begin{array}{cccc}
0&0&*&*\cr
0&1&*&*\cr
0&2&*&*\cr
1&*&0&*\cr
1&*&1&*\cr
1&*&2&*\cr
2&*&*&0\cr
2&*&*&1\cr
2&*&*&2
\end{array}\right).
$$
\caption{Examples of fractal matrices.}
\label{fig:examples_of_fractal_matrices}
\end{figure}

Matrices obtained from fractals by permutation of rows and columns will also be called fractal.

\begin{remark}\label{remark_on_permut_of_rows}
It is clear that permutation of rows of a star matrix defines another ordered, but the same unordered partition into subcubes.
\end{remark}

\begin{remark}
One could define as fractal a star matrix obtained from fractal by the replacement
in some of its columns of all numeric symbols in accordance with some permutation $\pi$ from $S_q$,
but it is not necessary to do this, because such a matrix will automatically be fractal by virtue of the definitions already given, since it can be obtained by permutation of rows and columns
of the original star matrix. Indeed, if the permutation $\pi$ of the numeric symbols is applied to the leftmost
column of the matrix $M_{q,m}$ in Fig.~\ref{fig:recursive_construction_of_fractal_matrix}, then the same
result can be achieved by permutations $\pi^{-1}$ of the horizontal and vertical stripes\footnote{The vertical stripes are counted without taking into account the leftmost column.}.
\end{remark}

It is easy to check that the fractal star matrix has the properties that will be listed later in this section.

\begin{property}\label{number_of_rows_in_fractal}
The number of rows in the fractal star matrix $M_{q,m}$ is $q^m$.
\end{property}

\begin{proof}
Follows by induction from the recursive definition of a fractal.
\end{proof}

\begin{property}\label{number_of_columns_in_fractal}
The number of columns in the fractal star matrix $M_{q,m}$ is $\frac{q^m-1}{q-1}$.
\end{property}

\begin{proof}
Induction on $m$. For $m=0$, the number of columns is $\frac{q^0-1}{q-1}=0$. Let the statement be true for $m-1$. Then for the parameter $m$, the number of columns by the definition of the fractal matrix is
$1+q\cdot \frac{q^{m-1}-1}{q-1}=\frac{q^m-1}{q-1}$, which is what was required to be proved.
\end{proof}

\begin{property}
The fractal star matrix $M_{q,m}$ defines a partition of the hypercube ${\bf Z}_q^{(q^m-1)/(q-1)}$ into $q^m$
subcubes of dimension $\frac{q^m-1}{q-1}-m$ each.
\end{property}

\begin{proof}
Follows immediately by induction from the definition of the recursive construction of the matrix $M_{q,m}$.
\end{proof}

\begin{observation}
The fractal star matrix, after inserting columns of only stars into it up to a total of $n$ columns, turns into a partition star matrix obtained by the sequence of
cuts presented in the second part of Section \ref{sec:asymp_lower_bound} to prove the asymptotic lower bound (\ref{eq:asymp_lower_bound}).
\end{observation}

\begin{property}\label{number_of_different_partitions_for_fractal}
The number of different unordered partitions defined by fractal matrices with parameters $q$ and $m$
is exactly $\left(\frac{q^m-1}{q-1}\right)!$.
\end{property}

\begin{proof}
We have already noted in Remark \ref{remark_on_permut_of_rows} that permuting rows does not yield a new
unordered partition. In Property \ref{number_of_columns_in_fractal} we have established that in
the fractal matrix with the specified parameters there are exactly $\frac{q^m-1}{q-1}$ columns. The fact that all permutations of columns yield different partitions of the hypercube can be understood by thinking about the process
of cutting the hypercube described in the second part of Section \ref{sec:asymp_lower_bound}, or we can
use the inductive procedure for constructing a fractal. Indeed, let the statement
be true for $m-1$, then for $m$ we have that the number of different partitions is
$$
\frac{q^m-1}{q-1}\cdot\binom{\frac{q^m-1}{q-1}-1}{\frac{q^{m-1}-1}{q-1},\dots,\frac{q^{m-1}-1}{q-1}}
\left(\left(\frac{q^{m-1}-1}{q-1}\right)!\right)^q=\left(\frac{q^m-1}{q-1}\right)!,
$$
which is what was required to be proved.
\end{proof}

We will say that the submatrix\footnote{The rows and columns of the submatrix in the star matrix $M$ do not have to be consecutive; in some subsequent figures, the rows and columns of the submatrix are shown in the star matrix $M$ as consecutive solely for ease of perception.}
$T$ of the star matrix $M$ is {\it a transfractal} if it is a
fractal\footnote{Naturally, we assume that the transfractal parameter $q$ is the same as that of the star matrix $M$ containing it.} and
all the columns that make up the submatrix $T$ outside the submatrix $T$ consist of only stars.

In the Figure \ref{fig:example_of_transfractal}, the transfractal is highlighted with a black frame.

\begin{figure}[ht]
\centering
$$
\begin{tikzpicture}[baseline=(current bounding box.center)]
\matrix (m) [matrix of nodes,nodes in empty cells,
             left delimiter=(,
             right delimiter=)] {
0&0&0&*&*&*\\
0&0&1&*&*&*\\
0&1&0&*&*&*\\
0&1&1&*&*&*\\
1&*&*&0&0&*\\
1&*&*&0&1&*\\
1&*&*&1&*&0\\
1&*&*&1&*&1\\
};
% Red frame
\draw[thick,red] (m-5-4.north west) rectangle (m-8-5.south west);
% Black frame
\draw[thick] (m-5-4.north west) rectangle (m-8-6.south east);
\end{tikzpicture}
$$
\caption{An example of a transfractal is highlighted with a black frame.}
\label{fig:example_of_transfractal}
\end{figure}

The number of rows in a transfractal will be called the {\it transfractal size.}
A column included in a transfractal, the size of which coincides with the size of the transfractal,
will be called the {\it leading column} of the transfractal. It is easy to see from the definition of a fractal that
any transfractal has exactly one leading column.
The rod of the leading column of a transfractal will also be called the {\it transfractal rod.}
In Fig.~\ref{fig:example_of_transfractal} the rod of the transfractal is highlighted in red.

The size of the transfractal will be called also as {\it the size of the transfractal rod}.

\begin{observation}
All numerical elements of the transfractal are covered in a non-intersecting manner
by the rod of the leading column and $q$ smaller transfractals.
\end{observation}

\begin{proof}
By fractal construction.
\end{proof}

\begin{observation}
Each fractal column is a leading column of some transfractal.
\end{observation}

\begin{proof}
By fractal construction.
\end{proof}

\begin{lemma}\label{rows_of_transfractal}
Let the star matrix $M$ of the partition of the hypercube ${\bf Z}_q^n$ into subcubes of the same dimension contains a submatrix $T$ that is a transfractal.
Then all rows of the submatrix $T$ outside the submatrix $T$ are identical as rows.
\end{lemma}

\begin{proof}
Proof by induction on the size of the transfractal. By the Property \ref{number_of_rows_in_fractal} of a fractal, the size of a transfractal is a power of $q$.

Let the size of the transfractal $T$ be $q$. Then by the Lemma \ref{equality_in_single_column}, the leading column $\vec{t}$ of the transfractal $T$ contains each number from ${\bf Z}_q$ exactly once. Consider a column $\vec{s}$ lying outside the transfractal $T$. By Lemma
\ref{equality_for_sums by_modulo_in_two_columns} the overlap size of columns $\vec{s}$ and $\vec{t}$ must be divisible by $q$. Hence, it is either $0$ or $q$. If the overlap size is $0$, then the rows of the submatrix $T$ in the column $\vec{s}$ contain only stars, and, therefore, all such rows coincide in the column $\vec{s}$. Let the overlap size of the columns $\vec{s}$ and $\vec{t}$ be $q$, i.e. the rows of the submatrix $T$ in the column $\vec{s}$ contain only numbers. Suppose that in this overlap in the column $\vec{s}$ there are two different numerical values:
$m_{r_1,s}\ne m_{r_2,s}$. We have already stated that $m_{r_1,t}\ne m_{r_2,t}$, so we can choose a permutation $\pi\in S_q$ such that $(m_{r_1,t}+\pi(m_{r_1,s}))\pmod{q}=(m_{r_2,t}+\pi(m_{r_2,s}))\pmod{q}$,
for example, by setting $\pi(m_{r_1,s})=m_{r_2,t}$, $\pi(m_{r_2,s})=m_{r_1,t}$.
However, this immediately leads to a contradiction with the statement of the Lemma
\ref{equality_for_sums by_modulo_in_two_columns}. Therefore, the assumption of two different numerical values in the column $\vec{s}$ in its overlap with the column $\vec{t}$ turned out to be incorrect. Therefore, all rows of the transfractal $T$ coincide in the column $\vec{s}$. Any column outside the transfractal $T$ could have been taken as a column $\vec{s}$, therefore the statement of the lemma is proved for the transfractal $T$ with the number of rows $q$, which is the basis of induction.

Let us now prove the inductive step. Let the statement of the lemma be true for transfractals of size $q^{l-1}$. Let us consider a transfractal with $q^l$ rows. By the fractal structure (Property \ref{number_of_rows_in_fractal}) and Lemma \ref{equality_in_single_column}, the leading column $\vec{t}$ of the transfractal contains each value from ${\bf Z}_q$ exactly $q^{l-1}$ times. Thus, the transfractal $T$ is divided into $q$ stripes with the same value in the column $\vec{t}$ within each stripe. By the transfractal structure, each of the stripes contains a leading column of size $q^{l-1}$ of a smaller transfractal. Hence, by the inductive hypothesis, for each of the $q$ stripes, all rows of this stripe coincide outside the transfractal $T$. Let us consider an arbitrary column $\vec{s}$ lying outside the transfractal $T$. From the above, within each of the $q$ stripes in the column $\vec{s}$ there is the same element (a number or a star). Hence, each pair of elements $(m_{r,t},m_{r,s})$ in the $q^l$ rows of the transfractal $T$ is duplicated $q^{l-1}$ times. Let us use the Lemma \ref{equality_for_sums by_modulo_in_two_columns}. Let us group the $q^{l-1}$ matching values in the columns $\vec{t}$ and $\vec{s}$ in each stripe into a single value. This will not violate the result of the analysis of the assertion of the Lemma
\ref{equality_for_sums by_modulo_in_two_columns} with respect to the number of matching values of sums modulo $q$. Thus, we have moved on to two columns of height $q$ each, for which we can repeat the arguments from the proof of the induction base of the lemma we are proving. We thus show that all rows of the transfractal $T$  coincide in the column $\vec{s}$. Since the choice of the column $\vec{s}$ outside the transfractal $T$ was arbitrary, we have proved the inductive step, and with it the entire lemma.
\end{proof}

\begin{corollary}\label{corollary_rows_of_transfractal}
Let $\vec{t}$ be a column included in the transfractal $T$, and $\vec{s}$ be a column external to $T$. Then the columns $\vec{t}$ and $\vec{s}$ either do not overlap, or the column $\vec{t}$ is inserted into the column $\vec{s}$, and in the latter case each number in the column $\vec{t}$ lies in the same row with the same number $a$ in the column $\vec{s}$.
\end{corollary}

\section{Operation of the bang of a star matrix}

We introduce the operation of the bang of a star matrix. We define the operation only for star matrices of A-primitive partitions into subcubes of the same dimension.

Let $M$ be a star matrix of an A-primitive partition into subcubes of the same dimension.
We select a column $\vec{i}$ and a value $a$ from ${\bf Z}_q$. We perform the following actions.

\begin{itemize}
\item{ 1. Delete the column $\vec{i}$.}

\item{ 2. Delete all rows in which the column $\vec{i}$ contained a numerical value different from $a$.}

\item{ 3. Duplicate each of the rows in which the column $\vec{i}$ contained the numerical value $a$ to $q$ identical rows.}

\item{ 4. For each of the resulting stripes of $q$ identical rows, add to the matrix a column in which outside the stripe there are only stars, and within the stripe there is each numerical value once.}

\item{ 5. If the matrix contains columns of only stars, then delete them.}
\end{itemize}

The described sequence of actions defines {\it the operation of the bang}\footnote{The term ''bang'' seems to us to be accurate for the process described. Indeed, during the bang, some columns and rows are destroyed, another part of the rows is split, and pieces of their contents fly away.}
of the star matrix $M$. Sometimes, for convenience, instead of the matrix bang, we will speak of the bang of the column $\vec{i}$, implying the same actions.

If the $i$th column contained $kq$ numbers, then when it banged, $k(q-1)$ rows were deleted from the matrix, but the same number were added due to a duplication, so the total number of rows remained the same. In the rows with stars in the $i$th column, numbers were neither deleted nor appeared, and in the rows obtained from the row with the number $a$ in the $i$th column, this number $a$ disappeared, but a number was added in the newly assigned column. Thus, the total number of numbers in all rows remained the same and the same as in the matrix $M$. It is easy to see that any two rows of the resulting matrix will contain different numbers in at least one column --- problems along the way arose only with duplicate rows, but they received different numerical values in the newly assigned column. Finally, due to the last operation, the matrix will not have columns of only stars. Thus, the bang results in an A-primitive star matrix, which is a partition matrix into the same number of subcubes of the same dimension as the original matrix $M$. However, the number of columns in the matrix may change.

Let us describe informally how a partition is transformed during a bang. First, the $i$th coordinate is actually removed from the space. Only those vectors remain that had the numerical value $a$ in the $i$th component. In this case, if the star pattern of a subcube contained a star in the $i$th component, then the dimension of the subcube decreases by $1$; subcubes with a numerical value in the $i$th component different from $a$ disappear completely; subcubes with a numerical value in the $i$th component equal to $a$ remain completely. Thus, the partition consists of subcubes of different dimensions, which may differ by $1$. However, adding new columns equalizes the dimensions again --- only stars are added to the star patterns of smaller subcubes, and each such star increases the dimension by $1$, and a number is added once to the star patterns of larger subcubes --- and due to this the dimensions are equalized.

Let us give an example of the bang of a star matrix. The bang process is shown in Fig.~\ref{fig:example_bang}.

\begin{figure}[ht]
\centering
$$
\scalebox{1.5}{
\def\0{\ccolor{$0$}}
\def\2{\ccolor{$*$}}
\def\1{\ccolor{$1$}}
\def\ARR{+(0.5,0.05) [->]edge +(1.5,0.4) +(0.5,-0.05) -- +(1.5,-0.4)}
\begin{tikzpicture}
\draw (0,7) \0\0\2\0 ;
\draw (0,6) \1\0\0\2 ;
\draw (0,5) \2\1\0\0 ;
\draw (0,4) \1\2\1\0 ;
\draw (0,3) \1\1\2\1 ;
\draw (0,2) \0\1\1\2 ;
\draw (0,1) \2\0\1\1 ;
\draw (0,0) \0\2\0\1 ;
\end{tikzpicture}
\raisebox{3.5em}{\ \ $\Longrightarrow\ \ $}
\begin{tikzpicture}
\draw (0,7) \0\0\2\0 \ARR;
\draw (0,6) \1\0\0\2 ;
\draw (0,5) \2\1\0\0 ;
\draw (0,4) \1\2\1\0 ;
\draw (0,3) \1\1\2\1 ;
\draw (0,2) \0\1\1\2 \ARR;
\draw (0,1) \2\0\1\1 ;
\draw (0,0) \0\2\0\1 \ARR;
%% now draw the lines
\draw (1,-0.5)--(1,7.5);
\draw (0.5,6)--(4.5,6);
\draw (0.5,4)--(4.5,4);
\draw (0.5,3)--(4.5,3);
\end{tikzpicture}
\raisebox{3.5em}{\ \ $\Longrightarrow\ \ $}
\begin{tikzpicture}
\draw (0,7) \0\2\0\0\2\2 ;
\draw (0,6) \0\2\0\1\2\2 ;
\draw (0,5) \1\0\0\2\2\2 ;
\draw (0,4) \1\1\2\2\0\2 ;
\draw (0,3) \1\1\2\2\1\2 ;
\draw (0,2) \0\1\1\2\2\2 ;
\draw (0,1) \2\0\1\2\2\0 ;
\draw (0,0) \2\0\1\2\2\1 ;
\end{tikzpicture} }
$$
\caption{An example of a bang of a star matrix.}
\label{fig:example_bang}
\end{figure}

As an example, we took the matrix from Rivest's paper \cite{Riv1974}. Here $q=2$.
We bang the left column with the value $0$.
The banged column is removed from the matrix; the $2$nd, $4$th and $5$th rows are also removed, as having numerical values in the exploded column different from $0$; at the same time, the $1$st, $6$th and $8$th rows, as having $0$ in the banged column, are split into two each, after which for each pair of cloned rows
we assign our own column to the right, in which the rows of this pair contain values from ${\bf Z}_2$ (i.e. $0$ and $1$) once, and stars in the remaining rows. Thus, a total of three new columns are assigned.
The bang did not produce columns of stars only, so there is no need to delete anything else.

We will bang up the resulting star matrix again.
The process of repeated bang is shown in Fig.~\ref{fig:example_bang_2}.

\begin{figure}[ht]
\centering
$$
\scalebox{1.15}{
\def\0{\ccolor{$0$}}
\def\2{\ccolor{$*$}}
\def\1{\ccolor{$1$}}
\def\ARR{+(0.5,0.05) [->]edge +(1.5,0.4) +(0.5,-0.05) -- +(1.5,-0.4)}
\begin{tikzpicture}
\draw (0,7) \0\2\0\0\2\2 ;
\draw (0,6) \0\2\0\1\2\2 ;
\draw (0,5) \1\0\0\2\2\2 ;
\draw (0,4) \1\1\2\2\0\2 ;
\draw (0,3) \1\1\2\2\1\2 ;
\draw (0,2) \0\1\1\2\2\2 ;
\draw (0,1) \2\0\1\2\2\0 ;
\draw (0,0) \2\0\1\2\2\1 ;
\end{tikzpicture}
\raisebox{3.5em}{\ \ $\Longrightarrow\ \ $}
\begin{tikzpicture}
\draw (0,7) \0\2\0\0\2\2 \ARR;
\draw (0,6) \0\2\0\1\2\2 \ARR;
\draw (0,5) \1\0\0\2\2\2 ;
\draw (0,4) \1\1\2\2\0\2 ;
\draw (0,3) \1\1\2\2\1\2 ;
\draw (0,2) \0\1\1\2\2\2 \ARR;
\draw (0,1) \2\0\1\2\2\0 ;
\draw (0,0) \2\0\1\2\2\1 ;
%% now draw the lines
\draw (1,-0.5)--(1,7.5);
\draw (0.5,5)--(6.5,5);
\draw (0.5,4)--(6.5,4);
\draw (0.5,3)--(6.5,3);
\end{tikzpicture}
\raisebox{3.5em}{\ \ $\Longrightarrow\ \ $}
\begin{tikzpicture}
\draw (0,7) \2\0\0\2\2\0\2\2 ;
\draw (0,6) \2\0\0\2\2\1\2\2 ;
\draw (0,5) \2\0\1\2\2\2\0\2 ;
\draw (0,4) \2\0\1\2\2\2\1\2 ;
\draw (0,3) \1\1\2\2\2\2\2\0 ;
\draw (0,2) \1\1\2\2\2\2\2\1 ;
\draw (0,1) \0\1\2\2\0\2\2\2 ;
\draw (0,0) \0\1\2\2\1\2\2\2 ;
%% now draw the lines
\draw (4,-0.5)--(4,7.5);
\end{tikzpicture}
\raisebox{3.5em}{\ \ $\Longrightarrow\ \ $}
\begin{tikzpicture}
\draw (0,7) \2\0\0\2\0\2\2 ;
\draw (0,6) \2\0\0\2\1\2\2 ;
\draw (0,5) \2\0\1\2\2\0\2 ;
\draw (0,4) \2\0\1\2\2\1\2 ;
\draw (0,3) \1\1\2\2\2\2\0 ;
\draw (0,2) \1\1\2\2\2\2\1 ;
\draw (0,1) \0\1\2\0\2\2\2 ;
\draw (0,0) \0\1\2\1\2\2\2 ;
\end{tikzpicture} }
$$
\caption{Repeated bang.}
\label{fig:example_bang_2}
\end{figure}

We bang the left column with the value $0$ again.
The banged column is removed from the matrix; the $3$rd, $4$th and $5$th rows are also removed, as having numerical values in the banged column different from $0$; at the same time, the $1$st, $2$nd and $6$th rows, as having $0$ in the banged column, are split into two, after which for each pair of cloned rows, we assign our own column to the right, in which the rows from this pair contain values from ${\bf Z}_2$ (i.e. $0$ and $1$) once, and stars in the remaining rows. Thus, a total of three new columns are assigned.
This results in a column of stars only --- we delete it from the matrix.

If you look closely, you will notice that the resulting $8\times 7$ matrix is fractal. As we will understand later, this is not accidental.

\section{Non-expandable star matrices}

We call a star matrix of an A-primitive partition into subcubes of the same dimension {\it expandable} if the number of columns in it increases under some bang. Accordingly, we call a star matrix of an A-primitive partition into subcubes of the same dimension {\it non-expandable} if the number of columns in it does not increase under any bang.

In this section, we will show that only fractal matrices are non-expandable.

\begin{lemma}\label{only_transfractls}
Let the A-primitive star matrix $M$ of the partition of the hypercube ${\bf Z}_q^n$ into subcubes of the same dimension be non-expandable. Then each of its columns has a size equal to a power of $q$, and the rod of this column is the rod of the leading column of some transfractal.
\end{lemma}

\begin{proof}
By Lemma \ref{equality_in_single_column} in any column of $M$ each numerical value occurs the same quantity of times, so the size of any column has the form $kq$. We will prove lemma by induction on $k$. For $k=1$, the rod of the column is a transfractal itself, that provides the basis for the induction.
Let us prove the inductive step. Consider a column $\vec{s}$ of size $kq$.
By the inductive hypothesis, the assertion of Lemma \ref{only_transfractls} is satisfied for all columns of smaller size; we prove it for the column $\vec{s}$.

We bang the column $\vec{s}$ with the value $0$. By definition of the bang, the column $\vec{s}$ will be removed from the matrix and $k$ new columns will be added by splitting the rows of the stripe $R_0$, consisting of all the rows of the matrix $M$ that had the value $0$ in the column $\vec{s}$. Also, it is possible that some number of columns will be removed, which will consist of only stars. Let us estimate the number of such columns (see~Fig.~\ref{fig:lemma_proof}).

\begin{figure}[ht]
\centering
\begin{tikzpicture}[scale=2.8,
    stripe/.style={minimum height=2cm, minimum width=12cm},
    brace/.style={decoration={brace,mirror,raise=5pt}, decorate},
    arrow/.style={->, thick},
    box/.style={draw, minimum width=0.2cm, minimum height=0.5cm}
]

% Labels outside stripes and braces
\node at (-7,5) {$R_0$};
\node at (-6.2,5) {$k$};
\draw[brace] (-5.7,6) -- (-5.7,4);

\node at (-7,3) {$R_1$};
\node at (-6.2,3) {$k$};
\draw[brace] (-5.7,4) -- (-5.7,2);

\node at (-7,-3) {$R_{q-1}$};
\node at (-6.2,-3) {$k$};
\draw[brace] (-5.7,-2) -- (-5.7,-4);

% Vector s with stars
\node at (-5,7) {$\vec{s}$};
\node at (-5,6.7) {$*$};
\node at (-5,6.4) {$*$};
\node at (-5,6.1) {$*$};

\node at (-3,6.7) {$*$};
\node at (-3,6.4) {$*$};
\node at (-3,6.1) {$*$};

\node at (-2,6.7) {$*$};
\node at (-2,6.4) {$*$};
\node at (-2,6.1) {$*$};

\node at (2,6.7) {$*$};
\node at (2,6.4) {$*$};
\node at (2,6.1) {$*$};

\node at (3,6.7) {$*$};
\node at (3,6.4) {$*$};
\node at (3,6.1) {$*$};

% Horizontal lines and column numbers
\draw (-5.8,6) -- (6,6);  % Top line R0
\draw (-5.8,4) -- (6,4);  % Bottom line R0/Top line R1
\draw (-5.8,2) -- (6,2);  % Bottom line R1
\draw (-5.8,-2) -- (6,-2);  % Top line R_{q-1}
\draw (-5.8,-4) -- (6,-4);  % Bottom line R_{q-1}

% Column numbers
\node at (-5,5.5) {$0$};
\node at (-5,5) {$\vdots$};
\node at (-5,4.5) {$0$};

\node at (-3,5.5) {$*$};
\node at (-3,5) {$\vdots$};
\node at (-3,4.5) {$*$};

\node at (-2,5.5) {$*$};
\node at (-2,5) {$\vdots$};
\node at (-2,4.5) {$*$};

\node at (2,5.5) {$*$};
\node at (2,5) {$\vdots$};
\node at (2,4.5) {$*$};

\node at (3,5.5) {$*$};
\node at (3,5) {$\vdots$};
\node at (3,4.5) {$*$};

\node at (-5,3.5) {$1$};
\node at (-5,3) {$\vdots$};
\node at (-5,2.5) {$1$};

\node at (2,3.5) {$*$};
\node at (2,3) {$\vdots$};
\node at (2,2.5) {$*$};

\node at (3,3.5) {$*$};
\node at (3,3) {$\vdots$};
\node at (3,2.5) {$*$};

\node at (-5,0) {$\vdots$};
\node at (-2,0) {$\vdots$};
\node at (-3,0) {$\vdots$};
\node at (2,0) {$\vdots$};
\node at (3,0) {$\vdots$};

\node at (-5,-2.5) {$q-1$};
\node at (-5,-3) {$\vdots$};
\node at (-5,-3.5) {$q-1$};

\node at (-3,-2.5) {$*$};
\node at (-3,-3) {$\vdots$};
\node at (-3,-3.5) {$*$};

\node at (-2,-2.5) {$*$};
\node at (-2,-3) {$\vdots$};
\node at (-2,-3.5) {$*$};

% Fragment in R1
%\node[box] at (-1,3.6) {};
%\node at (0,3.6) {$l_{1,1}$};
%\node at (-0.7,3.8) {$*$};  % Top star
%\node at (-0.7,3.4) {$*$};  % Bottom star

% Diagonal boxes in R1 (closer to column s)
\node[box] at (-3,3.6) {};
\node at (-2.5,3.6) {$l_{1,1}$};

\node[rotate=10] at (-2.5,3.1) {$\ddots$};
\node[box] at (-2,2.4) {};
\node at (-1.4,2.4) {$l_{1,p_1}$};

\node at (-2.0,3.8) {$*$};  % Top star
\node at (-2.0,3.4) {$*$};  % Bottom star

\node at (-3.0,2.6) {$*$};  % Top star
\node at (-3.0,2.2) {$*$};  % Bottom star

%\node[box] at (1,2.4) {};
%\node at (2,2.4) {$l_{1,p_1}$};
%\node at (1.3,2.6) {$*$};  % Top star
%\node at (1.3,2.2) {$*$};  % Bottom star

% Fragment in R_{q-1} (similar layout to R1)
\node[rotate=10] at (2.5,-2.9) {$\ddots$};

\node[box] at (2,-2.4) {};
\node at (2.7,-2.32) {$l_{q-1,1}$};
\node at (3.,-2.2) {$*$};  % Top star
\node at (3.,-2.6) {$*$};  % Bottom star

\node[box] at (3,-3.6) {};
\node at (4.0,-3.6) {$l_{q-1,p_{q-1}}$};
\node at (2.,-3.4) {$*$};  % Top star
\node at (2.,-3.8) {$*$};  % Bottom star

% Diagonal dots between fragments
\node[rotate=45] at (0,0) {$\vdots$};
%\node[rotate=45] at (1.5,0) {$\vdots$};
%\node[rotate=45] at (3,0) {$\vdots$};

% Bang 0 arrow at bottom
\draw[arrow] (-5,-4.3) -- (-5,-4);
\node at (-5,-4.6) {bang 0};

\end{tikzpicture}
\caption{Illustration for the proof of Lemma \ref{only_transfractls}.}
\label{fig:lemma_proof}
\end{figure}

Let all numbers in the column $\vec{t}$ of the matrix $M$ disappear during the bang. From the definition of a bang it is easy to see that the column $\vec{t}$ must be inserted into the column $\vec{s}$, and rows with numbers of the column $\vec{t}$ cannot have different numbers in the column $\vec{s}$ by  Corollary \ref{corollary_rows_of_transfractal}, since by the inductive hypothesis the column $\vec{t}$ is the leading column of some transfractal, and the column $\vec{s}$ is external with respect to this transfractal. Therefore, the rod of the column $\vec{t}$ must be in one of $q-1$ stripes $R_i$ of size $k$, where the $i$th stripe consists of all $k$ rows of the matrix $M$ that have the value $i$ in the column $\vec{s}$, $i=1,\dots,k-1$.
Let us consider the stripe $R_i$ and among all the rods of columns lying in it, let us choose the maximal by inclusion set $T_i$ of such rods of columns that are not inserted into other such rods. According to the selection condition, rods from $T_i$ are rods of leading columns of transfractals that are pairwise not inserted in each other. Therefore, by Corollary \ref{corollary_rows_of_transfractal}, rods from $T_i$ do not overlap pairwise. Let their sizes be $l_{i,1}, \dots, l_{i,p_i}$. From the pairwise non-overlapping it follows that
\begin{equation}\label{eq:inequality_for_stripe}
\sum\limits_{j=1}^{p_i}l_{i,j}\le k.
\end{equation}
By the inductive hypothesis, all rods from $T_i$ have a size equal to a power of $q$ and are rods of leading columns of some transfractals, all columns of which will also disappear under the bang. We know the number of columns in a transfractal from the definitions of a transfractal, a fractal, and the Property
\ref{number_of_columns_in_fractal}. Therefore, the quantity $N$ of columns that will be removed from the matrix during the bang can be estimated, given (\ref{eq:inequality_for_stripe}), as
\begin{equation}\label{eq:estimation_of cancelled_allstars_columns}
N=1+\sum\limits_{i=1}^{q-1}\sum\limits_{j=1}^{p_i}\frac{q^{\log_q l_{i,j}}-1}{q-1}=
1+\sum\limits_{i=1}^{q-1}\frac{\left(\sum\limits_{j=1}^{p_i}l_{i,j}\right)-p_i}{q-1}\le
1+\sum\limits_{i=1}^{q-1}\frac{k-1}{q-1}=k,
\end{equation}
moreover, from the inequality in (\ref{eq:estimation_of cancelled_allstars_columns}) it is easy to see that the equality
$N=k$ can be achieved only if all $p_i=1$ and all $l_{i,1}=k$, $i=1,\dots,q-1$, i.e. when all sets $T_i$ consist of a single rod, the size of which is equal to the number of rows in the stripe $R_i$.
Carrying out a bang of the column $\vec{s}$ with the value $1$, we obtain the same requirement for the stripe $R_0$.

It follows that the number of columns in the matrix $M$ will not increase under some bang of the column $\vec{s}$ only if the column $\vec{s}$ is the leading column of the transfractal and its size $kq$ is a power of the number $q$.
This completes the proof of the inductive step.
\end{proof}

\begin{corollary}\label{non_fractals_are_expandable}
Let $M$ be a non-expandable matrix. Then $M$ is a fractal matrix.
\end{corollary}

\begin{proof}
By Lemma \ref{only_transfractls}, each column of $M$ is a leading column of some
transfractal. Suppose that $M$ has two columns $\vec{t}_1$ and $\vec{t}_2$ that do not insert into each other or into any other columns. Then by Corollary \ref{corollary_rows_of_transfractal}, the columns $\vec{t}_1$ and $\vec{t}_2$ do not overlap.
Let us take a row $\vec{r}_1$ with a number in the column $\vec{t}_1$ and let us take a row $\vec{r}_2$ with a number in the column $\vec{t}_2$.

From the definition of a transfractal and the Corollary \ref{corollary_rows_of_transfractal} it follows that the rows $\vec{r}_1$ and $\vec{r}_2$ do not have a column in which these rows contain different numbers (indeed, if such a column $\vec{s}$ exists, then it is not inserted into either $\vec{t}_1$ or $\vec{t}_2$, and therefore is external with respect to the transfractals in which $\vec{t}_1$ and $\vec{t}_2$ are leading columns, but then, by the Corollary \ref{corollary_rows_of_transfractal}, both the column $\vec{t}_1$ and the column $\vec{t}_2$ are inserted into $\vec{s}$, that contradicts to our assumption). Thus, we have obtained a contradiction with Lemma \ref{lemma1} and the assumption made about the existence of columns $\vec{t}_1$ and $\vec{t}_2$.

This means that the matrix $M$ has a column $\vec{t}_1$ such that any other column is inserted into $\vec{t}_1$.
Then the column $\vec{t}_1$ must contain numbers in all rows, because if it contains a star in some row $r$, then, since $M$ does not have a row of only stars, there will be column $\vec{t}_2$ with a number in the row $r$, but then the column $\vec{t}_2$ is not inserted into the column $\vec{t}_1$, which contradicts the statement made above. Thus, the column $\vec{t}_1$ contains numbers in all rows and, therefore, the column $\vec{t}_1$ is the leading column of the entire matrix $M$, which, in turn, is a fractal matrix.
\end{proof}

\begin{observation}\label{only_fractal_is_nonexpandable}
A fractal matrix is non-expandable.
\end{observation}

\begin{proof}
This fact can be seen from the proof of the Lemma \ref{only_transfractls}. However, it is not necessary to do this, since any matrix other than a fractal one is expandable by the Corollary \ref{non_fractals_are_expandable}, and non-expandable
matrices must exist due to the boundedness by Theorem \ref{coord_thm}
of the number of columns in A-primitive matrices for fixed $q$ and $m$.
\end{proof}

\section{Main Theorems}

\begin{theorem}\label{tight_values}
We have exact values $N^{\rm coord}_q(m)=\frac{q^m-1}{q-1}$,
$c^{\rm coord*}_q\left(\frac{q^m-1}{q-1},m\right)=\left(\frac{q^m-1}{q-1}\right)!$.

\end{theorem}

\begin{proof}
Follows from Corollary \ref{non_fractals_are_expandable}, Observation \ref{only_fractal_is_nonexpandable}, and Property \ref{number_of_different_partitions_for_fractal}.
\end{proof}

The asymptotic formula (\ref{eq:main_asymptotics}), written out in {\bf Theorem \ref{th_asymp}}, follows
from Theorems \ref{general_form_of_asymptotics} and \ref{tight_values}.

\begin{corollary}
Let $q,m$ be fixed positive integer numbers, $q>1$. Then for $n\to\infty$
the number of partitions of the hypercube ${\bf Z}_q^n$ into subcubes of dimension at least $n-m$ is asymptotically equal to
$n^{\frac{q^m-1}{q-1}}$.
\end{corollary}

\begin{proof}
For the number of partitions into subcubes of dimension no less than $n-m$, a formula similar to (\ref{eq:coord_formula_modified}) is obviously valid, i.e. the question comes down to finding the
maximum number of columns in A-primitive matrices containing no more than $m$
numeric values in each row. However, if some row contains less than $m$ numerical values, then the corresponding
subcube can be split along a new coordinate into subcubes of smaller dimension, similar to what was done
during the operation of the bang, which leads to an increase in the columns of the matrix. Therefore, the maximum number of columns can only be achieved on a matrix containing exactly $m$ numbers in each row, i.e. on a matrix of A-primitive partitioning into subcubes of dimension $n-m$ each.
\end{proof}

\section{Conclusion}

Any star matrix, excepting a fractal, expands under some bang, whereas a fractal remains to be a fractal under any bang.

\end{document}